\documentclass[preprint,3p]{elsarticle}

\usepackage[T1]{fontenc}        
\usepackage{lmodern}            
\usepackage{anyfontsize}        
\usepackage{microtype}          
\AtBeginDocument{%
  \fontsize{8.5}{9.6}\selectfont 
  \linespread{0.95}\selectfont   
}

\usepackage{setspace}
\usepackage{enumitem}
\setlist{nosep, leftmargin=*}   

\usepackage{titlesec}
\titlespacing*{\section}{0pt}{6pt}{4pt}
\titlespacing*{\subsection}{0pt}{5pt}{3pt}

\usepackage[font=small,skip=1pt]{caption}

\usepackage{etoolbox}
\apptocmd{\thebibliography}{\setlength{\itemsep}{0pt}\small}{}{}

\usepackage{eurosym}
\usepackage{amsfonts}
\usepackage{amsmath}
\usepackage{amsthm}
\usepackage{amssymb}
\usepackage{mathrsfs}
\usepackage{xcolor} 
\usepackage[colorlinks=true,
linkcolor=magenta,
citecolor=cyan,
  urlcolor=orange
]{hyperref}
\usepackage{graphicx}
\usepackage[tableposition=top]{caption}
\usepackage{subcaption}
\usepackage{float}
\usepackage{color}
\usepackage{setspace}
\usepackage{placeins}
\usepackage{mathrsfs}
\usepackage{enumitem}
\usepackage{empheq}
\usepackage[utf8]{inputenc}
\newsavebox \foobox
\newlength{\foodim}

\graphicspath{{C:/Figures/}}
\newtheorem{theorem}{Theorem}

\newtheorem{definition}{Definition}

\newtheorem{lemma}{Lemma}

\newtheorem{proposition}{Proposition}
\newtheorem{remark}{Remark}

\numberwithin{equation}{section}

\begin{document}

\begin{frontmatter}

\title{Global Solutions of Coupled Nonlocal Parabolic Systems Arising from Reversible Chemical Reactions}

\author{Redouane Douaifia$^{a}$, Salem Abdelmalek$^{b}$, Mokhtar Kirane$^{c}$}
\address{(a) National Higher School of Telecommunications and Information and Communication Technologies -- Abdelhafid Boussouf, Oran, Algeria
\\
(b) Department of Mathematics, Laboratory (LAMIS), Echahid Cheikh Larbi Tebessi University, Tebessa, Algeria
\\
(c) Department of Mathematics, College of Computing and Mathematical Sciences, Khalifa University, P.O. Box: 127788, Abu
Dhabi, UAE
\\ 
Emails: redouane.douaifia@ensttic.dz; salem.abdelmalek@univ-tebessa.dz; mokhtar.kirane@ku.ac.ae

}
\begin{abstract}
A class of coupled time-space fractional reaction-diffusion systems derived from reversible chemical reactions over a bounded domain is investigated. Employing mainly an appropriate Lyapunov functional and an improved maximum principle, we demonstrate the global-in-time existence of strong solutions under some assumptions on the systems’ parameters.
\end{abstract}
\end{frontmatter}


\section{Introduction}
In this work, we investigate the global-in-time existence of nonnegative strong solutions for a class of coupled time-space fractional parabolic reaction-diffusion systems
\begin{equation}
\label{Main_system}
\begin{cases}
\mathcal{D}^\rho_{0 \vert t}u+d_u (-\Delta)^{\sigma_1} u= (\alpha_2-\alpha_1)(k_f u^{\alpha_1}v^{\beta_1}-k_b u^{\alpha_2}v^{\beta_2}) &  \mbox{ in }\Omega\times(0,\infty), \\
\mathcal{D}^\rho_{0 \vert t}v+d_v (-\Delta)^{\sigma_2} v= (\beta_2-\beta_1)(k_f u^{\alpha_1}v^{\beta_1}-k_b u^{\alpha_2}v^{\beta_2}) &  \mbox{ in }\Omega\times(0,\infty),\\
\lambda u + (1-\lambda)\nabla{u} \cdot \nu=\lambda v + (1-\lambda)\nabla{v} \cdot \nu=0 &  \mbox{ on }\Gamma \times(0,\infty),\\
u(\cdot,0)=u_0,\ v(\cdot,0)=v_0 &   \mbox{ in } \Omega ,
\end{cases}
\end{equation}  
which coming out of the reversible chemical reaction
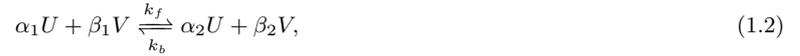
\begin{equation}
    \alpha_1 U+\beta_1 V\xrightleftharpoons[k_b]{k_f}\alpha_2 U+\beta_2 V,
\end{equation}  
where \(u:=u(x,t)\) and \(v:=v(x,t)\) denote the concentrations of the species \(U\) and \(V\), respectively. Here, \(\alpha_i, \beta_i \ge 0\) are the stoichiometric coefficients, \(k_f, k_b>0\) are the reaction rates, and \(d_u, d_v>0\) represent the diffusion coefficients. The domain \(\Omega \subset \mathbb{R}^N\) is a bounded smooth open set with \(\Gamma := \partial \Omega\) or \(\Gamma := (\mathbb{R}^N \setminus \Omega)\) , for \(1 \le N \in \mathbb{N}\). The time derivative operator $\mathcal{D}^\rho_{0 \vert t}$ is the Caputo fractional derivative of order $0< \rho \le 1$ (see \cite[Definition 2.1.1 and Remark 2.1.2]{gal2020fractional})
    \begin{equation}
		{\mathcal{D}}_{0 \vert t}^{\rho} u(x,t) := \frac{1}{\Gamma(1-\rho)} \int_{0}^{t} (t-\tau)^{- \rho} \partial_{\tau} u(x,\tau) \,d\tau,\quad x\in \Omega, \quad t>0.
    \end{equation}
    Furthermore, the spatial operator $(-\Delta)^\sigma$ (with $\sigma=\sigma_1$ or $\sigma=\sigma_2$) represents a spectral fractional Laplacian associated with either Dirichlet ($\lambda =1$) or Neumann ($\lambda =0$) homogeneous boundary conditions, which its definition given via the spectrum of the classical Laplacian operator $-\Delta$ (see \cite[Example 2.3.8]{gal2020fractional}). 
Let $\{\mu_k,e_k\}_{k\in\mathbb{N}}$ be  the eigenpairs of the Dirichlet or Neumann classical Laplacian. We also introduce the following fractional order Sobolev
space:
    \begin{equation}
		H^\sigma(\Omega):= \left\{w=\sum_{n=1}^{\infty} w_n e_n \in L^2(\Omega) \ :\  \|w\|_{H^\sigma(\Omega)}:=\sum_{n=1}^{\infty}\mu_n^\sigma w_n^2 < +\infty \right\},
    \end{equation}
    where $\sigma\in(0,1)$ and $w_n=\displaystyle\int_\Omega w(\mathbf{x}) e_n (x) \,dx$.
Then, the spectral fractional Laplacian is defined on the space $H^\sigma(\Omega)$ by
	\begin{equation}
    \label{FracLap_Spec}
		(-\Delta)^\sigma w(x):=\sum\limits_{n=1}^{\infty} \mu_{n}^\sigma w_{n}e_{n} (x),
	\end{equation}
    where its domain, for $\lambda \in \{0,1\}$, is given by
    \begin{equation}\label{Dom_Frac_Lap}
        \mathcal{D}\left((-\Delta)^\sigma\right)=\{w\in H^\sigma(\Omega) \ : \ \lambda w + (1-\lambda)\nabla{w} \cdot \nu=0\}.
    \end{equation}
    
Classical reaction–diffusion systems have been extensively used in various fields, including physics, biology, ecology, and chemistry (see \cite{Fife1979,Murray2001,Abdelmalek2019}). However, as research has advanced, it has become clear that these models are inadequate to fully capture the complexity of certain systems.\\
Let us mention that, a great attention has been devoted to the global existence of classical reaction-diffusion models (i.e. $\rho=\sigma_1=\sigma_2=1$) of the type described in system \eqref{Main_system} over the past 35 years by numerous authors, Notable contributions include those by Kirane \cite{Kirane1990}, Kouachi \cite{Kouachi2001,Kouachi2011}, Pierre \cite{Pierre2010}, Craciun \textit{et al.} \cite{Craciun2017}, and see also references therein.

On the other hand, diffusion processes originate from random molecular motion, and random walk models serve as fundamental frameworks for describing diffusion phenomena at the microscopic scale. From these models, one can derive deterministic fractional diffusion equations that accurately capture anomalous diffusion behavior. Moreover, since the past history of a system can influence its dynamic behavior, it is appropriate to incorporate these memory effects into the model through a fractional derivative with respect to time (see \cite{Caputo1999}; for more details, cf. \cite[Appendix C]{gal2020fractional}).\\
Consequently, it is interest to study reaction-diffusion systems that involve fractional derivatives with respect to both time and space, which can be expressed in terms of nonlocal operators (i.e., the support of a function is not preserved).

To the best of our knowledge, only a few studies have addressed the global existence of fractional time--space reaction--diffusion systems on bounded domain (see, e.g., Gal \textit{et al.} \cite{gal2020fractional}, Ahmad \textit{et al.} \cite{Ahmad2020}, Alsaedi \textit{et al.} \cite{Alsaedi2021}, Kirane \textit{et al.} \cite{Kirane2025}, and the references therein). More recently, Daoud \textit{et al.}~\cite{Daoud2024,Daoud2025} achieved significant advances in this context, for the case $\rho = 1$ (the classical time derivative) and for fractional Laplacians in both the spectral and Riesz formulations. However, their results do not cover the situation described by system~\eqref{Main_system}.

In contrast, the global-in-time existence of solutions to system~\eqref{Main_system} has not yet been established. In the present work, based on the local existence and positivity of solutions and under some assumptions on the system's parameters, we derive a~priori $L^\infty$-estimates using a suitable Lyapunov functional and an improved maximum principle. These estimates yield a global existence result for system~\eqref{Main_system}.
\subsection{Main result}
\begin{theorem}
\label{Main_Theorem}
    Let $0<\rho \leq 1$, $\sigma_1,\sigma_2\in (0,1)$, and $u_0,v_0\in C(\Bar{\Omega})$, with $0\leq u_0,v_0\leq \Lambda <+\infty$ for a.e. $x$ in $\Omega$, such that $u_0,v_0\not\equiv 0$. Moreover, let $\lambda\in \{0,1\}$, $(d_u,d_v,k_f,k_b)\in (0,+\infty)^4$, and $(\alpha_i,\beta_i)\in [0,+\infty)^2$ for $i=1,2$. Then, the system \eqref{Main_system} admits a unique nonnegative global strong solution, provided that one of the following assumptions is satisfied:
    \begin{enumerate}
  \renewcommand{\labelenumi}{(\roman{enumi})}
  \item $0<\alpha_1<\alpha_2$, $0<\beta_2<\beta_1$, $\alpha_1+\beta_1>\alpha_2+\beta_2$, and $k_f=k_b$.\label{Main_Theor_A1}
  \item $0<\alpha_2<\alpha_1$, $0<\beta_1<\beta_2$, $\alpha_1+\beta_1<\alpha_2+\beta_2$, and $k_f=k_b$.\label{Main_Theor_A2}
  \item $\alpha_1=\alpha_2$ whatever are $\beta_1$ and $\beta_2$.
  \item $\beta_1=\beta_2$ whatever are $\alpha_1$ and $\alpha_2$.
  \item \label{case5} $\rho=1$, $\sigma_1\leq \sigma_2$, $0<\alpha_1<\alpha_2$, $0<\beta_1<\beta_2$, and $\alpha_1+\beta_1\leq 1$.
  \item \label{case6} $\rho=1$, $\sigma_1\leq \sigma_2$, $0<\alpha_2<\alpha_1$, $0<\beta_2<\beta_1$, and $\alpha_2+\beta_2\leq 1$.
\end{enumerate}
\end{theorem}
\begin{remark}
\label{Rmrk_Main_Th_FracLap_Riez}
The same results as those presented in Theorem \ref{Main_Theorem} (with $\lambda=1$, $\Gamma = (\mathbb{R}^N \setminus \Omega)$, and $\sigma_1=\sigma_2$ in the cases (v)-(vi)) for system \eqref{Main_system} can be obtained by employing the operator $(-\Delta)^{\sigma}$, which denotes the fractional Laplacian introduced by M.~Riesz in~\cite{Riesz1938} and defined for any $\phi \in \mathscr{S}(\mathbb{R}^N)$ by (see \cite{Nezza2012}):
\begin{equation}\label{FracLap_Riez}
    (-\Delta)^{\sigma}\phi(x) := a_{N,\sigma}\, \mathrm{P.V.} \int_{\mathbb{R}^{N}} \frac{\phi(x)-\phi(y)}{\|x-y\|^{N+2\sigma}}\, dy, \qquad \sigma \in (0,1),
\end{equation}
where $\mathrm{P.V.}$ stands for the Cauchy principal value, $\|\cdot\|$ denotes the Euclidean norm in $\mathbb{R}^N$, and $a_{N,\sigma}$
is a normalization constant chosen so that the following identities hold:
\[
\lim_{\,\sigma \to 0^+} (-\Delta)^{\sigma}\phi = \phi
\quad \text{and} \quad
\lim_{\,\sigma \to 1^-} (-\Delta)^{\sigma}\phi = -\Delta \phi.
\]
\end{remark}
\section{Preliminaries}
	\subsection{Definitions}

	First, we recall some basic definitions from fractional calculus (see e.g. \cite{gal2020fractional,KilbasSrivastavaTrujillo}). Throughout the remainder, we consider $T>0$.
        \begin{definition}
            Let $f\in L^{1}([0,T])$.  The left Riemann-Liouville fractional integral of order $\rho>0$ is defined by
                \begin{align}
                        (I_{0|t}^\rho f)(t) &:= \frac{1}{\Gamma(\rho)} \int_0^t \frac{f(\tau) }{(t - \tau)^{1-\rho}} \, d\tau, 
                \end{align}
        \end{definition}
        
	
    \begin{definition}
        The \textit{two parameters Mittag-Leffler function} $E_{z_1,z_2}(z)$ is given by
            \begin{equation}
                E_{z_1,z_2}(z) = \sum_{k=0}^\infty \frac{z^k}{\Gamma(z_1 k +z_2)}, \quad z,z_1,z_2\in \mathbb{C};\ Re(z_1),Re(z_2)>0 .
            \end{equation}
    \end{definition}

    \begin{definition}\label{Def_P}
       Let $\rho\in (0,1)$. The operator $\mathcal{P}_\rho(t):H^\sigma(\Omega)\to H^\sigma(\Omega)$ is given by
        \begin{equation}
            \mathcal{P}_\rho(t) w:=\alpha t^{\rho-1} \int_0^{\infty} \tau \Phi_\rho(\tau) \mathcal{S}\left(\tau t^\rho\right) w \ d \tau,
        \end{equation}
where $\mathcal{S}(t)$ is  the strongly continuous semigroup generated by the operator $(-\Delta)^\sigma$ and $\Phi_\rho(z)$ is the Wright type function
    \begin{equation*}
        \Phi_\rho(z):=\sum_{n=0}^{\infty} \frac{(-z)^n}{n ! \Gamma(-\rho n+1-\rho)}, \quad z \in \mathbb{C}.
    \end{equation*}
    \end{definition}
    \subsection{Auxiliary Results}
	In this subsection, we present several lemmas that are essential for establishing the main result.
            \begin{lemma} \cite[Lemma 2.22]{KilbasSrivastavaTrujillo}\label{Lemma_frac_property}
			For a continuously differentiable function $f$, the following relationship between fractional integral and Caputo fractional derivative holds 
			\begin{equation}
					I_{0|t}^\alpha \mathcal{D}_{0|t}^{\alpha } f(t) = f(t) - f(0) \label{eqn J^aD^a}.
                \end{equation}
		\end{lemma}

		\begin{lemma}\cite[Lemma 3]{Kirane2025}
        \label{Lemma_chain_convexD}
			If $\varphi \in C^1(\mathbb{R})$ is a convex function and $x\in C([0,T])\cap C^1((0,T])$, then 
				\begin{equation*}
					\mathcal{D}^\alpha_{0 \vert t} \varphi(x)(t) \le \varphi'(x(t))\mathcal{D}^{\alpha}_{0 \vert t} x(t),
					\qquad0<\alpha\le1.
				\end{equation*}
		\end{lemma}
     We need the following result on weakly singular Gronwall's inequality.     
        \begin{lemma} \cite[Lemma 6.19]{diethelm2010analysis}
        \label{ineq_weakly_Gronwall}
            Let $\alpha,C\in \mathbb{R}_+$. If $\psi:[0,T]\to \mathbb{R}$ is a nonnegative continuous function that satisfies
                \begin{equation*}
                    \psi(t) \le \psi(0)+\frac{C}{\Gamma(\alpha)} \int_0^t (t-s)^{\alpha-1} \psi(s) \, ds,
                \end{equation*}
            for all $t\in [0,T]$, then
                \begin{equation*}
                    \psi(t) \le \psi(0) E_{\alpha,1}(Ct^\alpha),\quad \text{for}\quad 0\le t\le T.
                \end{equation*}
        \end{lemma}
        We recall the boundedness of the operator $\mathcal{P}_{\rho}$ (see Definition \ref{Def_P}) for $\rho \in (0,1)$.
\begin{lemma}\cite[Proposition 2.1.3]{gal2020fractional}\label{Lemma_P_bound}
    There exists $C>0$ such that for all $t>0$ and $w\in H^{\sigma}(\Omega)$, we have
    \begin{equation}
        \|t^{1-\rho}P_{\rho}(t)w\|_{H^{\sigma}(\Omega)} \le C\|w\|_{H^{\sigma}(\Omega)}.
    \end{equation}
\end{lemma}
The following inequalities will be useful in the sequel.
        \begin{lemma}\cite[Lemma 6.2]{Bonforte2022}
        \label{Stroock_Varopoulos}
            Let $\sigma\in (0,1)$ and $p> 1$. For \eqref{FracLap_Spec} and nonnegative $u\in L^p(\Omega)$ with $(-\Delta)^{\sigma}u \in L^p(\Omega)$, we have
\begin{align}\label{Lemma_Strook&Varopulous_ineq}
            \int_{\Omega} u^{p-1} (-\Delta)^{\sigma}u \ dx 
                    \ge \frac{4(p-1)}{p^2} \int_{\Omega} \left|(-\Delta)^{\frac{\sigma}{2}} u^{\frac{p}{2}}\right|^2 dx. 
                \end{align}
        \end{lemma}
        \begin{remark}
        For the operator \eqref{FracLap_Riez}, the Stroock–Varopoulos inequality \eqref{Lemma_Strook&Varopulous_ineq} holds (see \cite[Formula (B7)]{Bonforte2017}).
        \end{remark}
\begin{lemma}
    \label{positiv_inequ_lemma}
Let $x, y \ge 0$, $r>0$ and $s>t>0$. Then
\[
(x^s - y^t)(x^{\frac{rs}{t}} - y^r) \ge 0.
\]
\end{lemma}
\begin{proof}
    The demonstration is straightforward by setting \( a := x^{s/t} \) and \( b := y \).
\end{proof}
We also have the following maximum principle for the time and space fractional heat equation.
    \begin{lemma}
    \label{Lemma_Maximum_Principle}
    Let $\rho\in (0,1]$, $\sigma \in (0,1)$, $d>0$, $\psi_0\in L^{\infty}(\Omega)$ such that $\psi_0\ge 0$, and $\psi$ be the solution to the following problem
    \begin{equation}
\begin{cases}
\mathcal{D}^\rho_{0 \vert t}\psi+d (-\Delta)^{\sigma} \psi \le 0 &  \mbox{ in }\Omega\times(0,T), \\
\lambda \psi + (1-\lambda)\nabla{\psi} \cdot \nu= 0 &  \mbox{ on }\Gamma\times(0,T),\\
\psi(\cdot,0)=\psi_0 &   \mbox{ in } \Omega ,
\end{cases}
\end{equation}
where the operator $(-\Delta)^{\sigma}$ reepresents either \eqref{FracLap_Spec} (with $\lambda\in \{0,1\}$ and $\Gamma=\partial \Omega$), or \eqref{FracLap_Riez} (with $\lambda=1$ and $\Gamma = (\mathbb{R}^N \setminus \Omega)$). Then, we have
\begin{equation}
    \|\psi(.,t)\|_{L^{\infty}(\Omega)} \le \|\psi_0\|_{L^{\infty}(\Omega)}, \quad \text{for } 0<t<T.
\end{equation}
    \end{lemma}
    \begin{proof}
We follow the same lines as in \cite[Theorem~3]{AAK2}, with adaptations to account for the inequality case and the use of the Caputo fractional derivative instead of the Riemann–Liouville one, which affects the treatment of the initial condition.
    \end{proof}
Since the nonlinearities of system~\eqref{Main_system} are locally Lipschitz continuous and quasi-positive, the local existence and nonnegativity of strong solutions (see \cite[Definition~3.2.1]{gal2020fractional}) are well known (see e.g. \cite[Theorem~3.2.2, Subsection~3.6, Theorem~4.1.3]{gal2020fractional} and \cite[Sections~3 and~4]{Kirane2025}); for completeness, we recall their statement below.
\begin{proposition}\label{LocalExistence}
 	Assuming $u_0,v_0\ge 0$, with $u_0,v_0\not\equiv 0$, and $u_0,v_0\in C(\Bar{\Omega})$. Then, system \eqref{Main_system} possesses a local nonnegative strong solution on a maximal interval $[0,T_{\max})$. Moreover,
    \begin{equation}
    \label{Global_Ex_Carac}
        \text{if}\quad T_{\max}<+\infty, \quad\text{then} \lim_{t\nearrow T_{\max}}\left(\|u(\cdot,t)\|_{L^\infty(\Omega)}+\|v(\cdot,t)\|_{L^\infty(\Omega)}\right) =+\infty
    \end{equation}
\end{proposition}
\begin{remark}
According to the contrapositive of the statement~\eqref{Global_Ex_Carac}, to establish the global existence of strong solution to system~\eqref{Main_system}, it is sufficient to derive an a priori estimate of the form
\begin{equation}\label{GlobalExist_Estimate}
 \forall t\in [0,T_{\max}), \quad \|u(\cdot,t)\|_{L^\infty(\Omega)}+\|v(\cdot,t)\|_{L^\infty(\Omega)} <\mathcal{F}(t),  
 \end{equation}
 where $\mathcal{F} :[0,+\infty)\rightarrow [0,+\infty)$ is a continuous function.
\end{remark}
\section{Proof of the main result}
\begin{proof}[Proof of Theorem \ref{Main_Theorem}] Let $p,q>1$ and $\delta_1,\delta_2>0$. Consider the following Lyapunov functional
\begin{equation}
    \mathcal{L}(t)=\int_{\Omega} (\delta_1 u^p + \delta_2 v^q) \,dx.
\end{equation}
Then, by using Lemma \ref{Lemma_chain_convexD}, we get
\begin{eqnarray*}
\mathcal{D}^\rho_{0 \vert t}\mathcal{L}(t)
&\leq& -\int_{\Omega} \left( \delta_1 d_u p u^{p-1} (-\Delta)^{\sigma_1} u + \delta_2 d_v q v^{q-1} (-\Delta)^{\sigma_2} v\right) \,dx \\
&&+\int_{\Omega} \left[\left(\delta_1 p (\alpha_2-\alpha_1) u^{p-1}  +\delta_2 q (\beta_2-\beta_1) v^{q-1}  \right)(k_f u^{\alpha_1}v^{\beta_1}-k_b u^{\alpha_2}v^{\beta_2}) \right] \,dx \\
&=:&\mathcal{I}+\mathcal{J}.
\end{eqnarray*}%
Based on Stroock-Varopoulos inequality, Lemma \ref{Stroock_Varopoulos}, we obtain
\begin{equation}
    \mathcal{I}\leq -\int_{\Omega} \left(  \frac{4\delta_1 d_u (p-1)}{p}  \left|\displaystyle(-\Delta)^{\frac{\sigma_1}{2}} u^{\frac{p}{2}}\right|^2 + \frac{4\delta_2 d_v (q-1)}{q}  \left|\displaystyle(-\Delta)^{\frac{\sigma_2}{2}} v^{\frac{q}{2}}\right|^2  \right) \,dx \leq 0.
\end{equation}
On the other hand, we shall consider the two assumption cases.\\
(\textit{i}) If $0<\alpha_1<\alpha_2$, $0<\beta_2<\beta_1$, and $\alpha_1+\beta_1>\alpha_2+\beta_2$. Set $\widehat{\alpha}=\alpha_2-\alpha_1$, $\widehat{\beta}=\beta_1-\beta_2$, $k_f=k_b=k$, we moreover consider $\delta_1=(p\widehat{\alpha})^{-1}$ and $\delta_2=(q\widehat{\beta})^{-1}$. Therefore, we can let $p > 1$ be sufficient large to render $q = \frac{(p-1)\widehat{\beta}}{\widehat{\alpha}}+1 > 1$, and due to Lemma \ref{positiv_inequ_lemma}, we get
\begin{equation}
\mathcal{J} = -\int_{\Omega} \left[k\left(v^{q-1}-u^{p-1}  \right)( v^{\widehat{\beta}}- u^{\widehat{\alpha}})u^{\alpha_1}v^{\beta_2}  \right] \,dx \leq 0.
\end{equation}
(\textit{ii}) If $0<\alpha_2<\alpha_1$, $0<\beta_1<\beta_2$, and $\alpha_1+\beta_1<\alpha_2+\beta_2$. Set $\widetilde{\alpha}=\alpha_1-\alpha_2$, $\widetilde{\beta}=\beta_2-\beta_1$, $k_f=k_b=k$, we further consider $\delta_1=(p\widetilde{\alpha})^{-1}$ and $\delta_2=(q\widetilde{\beta})^{-1}$. Therefore, we can let $p > 1$ be sufficient large to render $q = \frac{(p-1)\widetilde{\beta}}{\widetilde{\alpha}}+1 > 1$, and due to Lemma \ref{positiv_inequ_lemma}, we get
\begin{equation}
\mathcal{J} = -\int_{\Omega} \left[k\left(v^{q-1}-u^{p-1}  \right)( v^{\widetilde{\beta}}- u^{\widetilde{\alpha}})u^{\alpha_2}v^{\beta_1}  \right] \,dx \leq 0.
\end{equation}%
Accordingly, in either cases (\textit{i}) or (\textit{ii}), we have
\begin{equation}
\label{DrhoL_Negative}
    \mathcal{D}^\rho_{0 \vert t}\mathcal{L}(t) \leq 0,\quad \text{for all}\quad t\in (0,T_{max}).
\end{equation}
Hence, by using Lemma \ref{Lemma_frac_property}, we deduce that
\begin{equation}
\label{L_UpperBound}
    \mathcal{L}(t) \leq \mathcal{L}(0) \leq \left|\Omega \right| (\delta_1 \Lambda^p+\delta_2 \Lambda^q),\quad \text{for all}\quad t\in (0,T_{max}).
\end{equation}
Now, we deal just with the case (\textit{i}) (and the argument for case (\textit{ii}) is similar). From \eqref{L_UpperBound}, we have
\begin{equation}
    \int_{\Omega} \delta_1 u^p \,dx \leq  \mathcal{L}(0) \leq \left|\Omega \right| (\delta_1 \Lambda^p+\delta_2 \Lambda^q),\quad \text{for all}\quad t\in (0,T_{max}).
\end{equation}
Then, since $q-1 = \frac{(p-1)\widehat{\beta}}{\widehat{\alpha}}$ implies $\frac{q}{p}< \frac{\hat{\beta}}{\hat{\alpha}}$, we get
\begin{equation}
    \| u(.,t)\|_{L^p(\Omega)} \leq \left|\Omega \right|^{\frac{1}{p}} \left( \Lambda^p+\frac{p\widehat{\alpha}}{q\widehat{\beta}} \Lambda^q\right)^{\frac{1}{p}}\leq \left|\Omega \right|^{\frac{1}{p}}\max(2,2\Lambda^{\frac{\hat{\beta}}{\hat{\alpha}}}),\quad \text{for all}\quad t\in (0,T_{max}).
\end{equation}
Therefore, by letting $p \to \infty$, we obtain
\begin{equation}
\label{c1_Bound_u}
    \| u(.,t)\|_{L^{\infty}(\Omega)} \leq \Lambda_u:=\max(2,2\Lambda^{\frac{\hat{\beta}}{\hat{\alpha}}}),\quad \text{for all}\quad t\in (0,T_{max}).
\end{equation}
Similarly, by letting $q \to \infty$, it follows that
\begin{equation}
\label{c1_Bound_v}
    \| v(.,t)\|_{L^{\infty}(\Omega)} \leq \Lambda_v:=\left(2\left(\frac{\widehat{\beta}}{\widehat{\alpha}}\right)^2+1\right) \max (1, \Lambda),\quad \text{for all}\quad t\in (0,T_{max}).
\end{equation}
Thus, from \eqref{c1_Bound_u} and \eqref{c1_Bound_v} we get the global existence of \eqref{Main_system}.\\
(\textit{iii}) If $\alpha_1 = \alpha_2$, regardless of the values of $\beta_1$ and $\beta_2$. Let $T\in (0,T_{max})$. The first equation of \eqref{Main_system}, for its part, yields
\[
\|u(.,t)\|_{L^\infty(\Omega)} \leq \|u_0\|_{L^\infty(\Omega)},\quad \text{for all}\quad t\in (0,T).
\]
On the other hand, without loss of generality, we assume $\beta_2 > \beta_1$. Thus, the second equation of \eqref{Main_system} yields
\begin{equation}
\label{v_boundEq}
    \mathcal{D}^\alpha_{0 \vert t}v+d_v (-\Delta)^{\sigma_2} v\leq C_v(1+v) \quad  \mbox{ in }\Omega\times(0,T).
\end{equation}
where $C_v:=C_v(\Gamma,\beta_1,\beta_2,k_f,k_b)>0$.
Let $\overline{v}$ be the solution of the following problem
\begin{equation}
\label{vBar_system}
\begin{cases}
\mathcal{D}^\alpha_{0 \vert t}\overline{v}+d_v (-\Delta)^{\sigma_2} \overline{v}= C_v (1+v) &  \mbox{ in }\Omega\times(0,T). \\
\lambda \overline{v} + (1-\lambda)\nabla{\overline{v}} \cdot \nu=0 &  \mbox{ on }\Gamma\times(0,T).\\
\overline{v}(\cdot,0)=0 &   \mbox{ in } \Omega .
\end{cases}
\end{equation}
Then, the solution $\overline{v}$ of \eqref{vBar_system}, satisfies
\begin{equation}
    \overline{v}(x,t)=C_v\int_0^t \mathcal{P}_{\alpha} (t-s) \left[ 1+v(x,s)\right] \,ds
\end{equation}
By subtracting the first equation of \eqref{vBar_system} from \eqref{v_boundEq}, we obtain
\begin{equation}
    \mathcal{D}^\alpha_{0 \vert t}(v-\overline{v})+d_v (-\Delta)^{\sigma_2} (v-\overline{v})\leq 0, \quad  \mbox{ in }\Omega\times(0,T).
\end{equation}
Thanks to Lemma \ref{Lemma_Maximum_Principle}, the second triangular inequality, and Lemma \ref{Lemma_P_bound}, we derive
\begin{equation}
    \|v(.,t)\|_{L^\infty(\Omega)} \leq \|v_0\|_{L^\infty(\Omega)}+ C_v \int_0^t (t-s)^{\alpha-1} 
\Bigl[ 1 + \|v(.,s)\|_{L^\infty(\Omega)} \Bigr] ds,\quad \text{for all}\quad t\in (0,T).
\end{equation}
By applying weakly singular Gronwall’s inequality in Lemma \ref{ineq_weakly_Gronwall}, we obtain
\[
\|v(.,t)\|_{L^\infty(\Omega)} \leq \widetilde{\mathcal{F}}(t),\quad \text{for all}\quad t\in (0,T),
\]
where $\widetilde{\mathcal{F}}:[0,+\infty)\to[0,+\infty)$ is a continuous function.
This implies $T_{\max} = +\infty$.\\
(\textit{iv}) This case can be handled similarly to (\textit{iii}).\\
(\textit{v}) This case follows directly from \cite[Theorem 3.2]{Daoud2025} (for the case mentioned in Remark \ref{Rmrk_Main_Th_FracLap_Riez}, see \cite[Theorem 4.2]{Daoud2024}), obtained by applying Young’s inequality and choosing
\begin{equation}
    \mathcal{Q} =
    \begin{pmatrix}
        1 & 0 \\
        1 & 1
    \end{pmatrix},\quad
    \text{and} \quad 
    \textbf{b} =
    \begin{pmatrix}
        k_f(\alpha_2-\alpha_1) \\
        k_f\left( \alpha_2-\alpha_1+\beta_2-\beta_1\right)
    \end{pmatrix}.
\end{equation}
(\textit{vi}) The argument for this case is analogous to (\textit{v}).\\
\end{proof}
\end{document}